\documentclass{amsart}
\usepackage{amsfonts}
\usepackage{amsmath}
\usepackage{amssymb}
\usepackage{graphicx, comment}
\usepackage{float}

\usepackage{tikz}
\usetikzlibrary{decorations.markings}
\usetikzlibrary{shapes}
\tikzstyle{vertex}=[circle, draw, inner sep=0pt, minimum size=6pt]

\newcommand{\vertex}{\node[vertex]}
\newtheorem{definition}{\bf Definition}

\newtheorem{lemma}{\bf Lemma}
\newtheorem{theorem}{\bf Theorem}

\newtheorem{remark}{\bf Remark}
\newtheorem{corollary}{\bf Corollary}

\begin{document}

\title[Graphs with characteristic-dependent well-covered dimension]{A note on graphs with characteristic-dependent well-covered dimension}
\author{J. Burdick}
\address{Humboldt State University. Arcata, CA 95521.}
\email{jab320@humboldt.edu}
\author{O. Vega}
\address{California State University, Fresno. Fresno, CA 93740.}
\email{ovega@csufresno.edu}

\thanks{Research partially supported by NSF (DMS \#1156273) during the REU program at California State University, Fresno.}
\subjclass[2010]{Primary 05C69; Secondary 05C50, 05C22}
\keywords{Well-covered dimension, maximal independent sets}

\begin{abstract}
We investigate graphs that have characteristic-dependent well-covered dimension and show how more of these graphs can be constructed from known ones.
\end{abstract}

\maketitle

Our objective is to find the smallest graphs in which a given parameter of the graph depends on the characteristic of the field used to compute such parameter. In order to do this we will first introduce the reader to all necessary concepts, and notation, to later provide examples of what we believe are the smallest graphs with this property. 

\section{Preliminaries}

In this paper, all graphs will be finite, simple and undirected. As it is customary, for any given graph $G$, we will use $V(G)$ to denote the set of vertices of $G$ and $E(G)$ to denote the set of edges of $G$. The cardinality of $V(G)$ is called the \textbf{order} of $G$. When two vertices are connected by an edge we will say they are \textbf{adjacent}, we will also say that either vertex is \textbf{incident} with the edge connecting them. We refer the reader to Bondy and Murty's book \cite{BM} for a more complete introduction to concepts in graph theory. 

\begin{definition}
Let $G$ be a graph. $M$ is an \textbf{independent set} of $G$ if $M \subseteq V(G)$ and no pair of vertices in $M$ are adjacent.  An independent set $M$ of $G$ is said to be \textbf{maximal} if it is not properly contained in any other independent set of $G$.
\end{definition}

\begin{definition}
Let $\mathbb{F}$ be a field and $G$ a graph. A \textbf{well-covered weighting} of $G$ is a function $w : V(G) \to \mathbb{F}$ such that \ $\sum\nolimits_{v \in M} w(v)$ is constant for all maximal independent sets $M$ of $G$. 
\end{definition}

There is an abundant amount of research about graphs having all maximal independent sets with the same cardinality; these graphs are said to be well-covered (see  Plummer \cite{P70} and  \cite{P93}). Note that a well-covered graph may be also defined as a graph that admits all constant functions as well-covered weightings. Hence, the study of well-covered weightings is an attempt to  generalize the study of well-covered graphs. 

\begin{definition}
Let $G$ be a graph and $\mathbb{F}$ a field. The $\mathbb{F}$-vector space of all well-covered weightings of $G$ is called the \textbf{well-covered space} of $G$ (relative to $\mathbb{F}$). \\
The \textbf{well-covered dimension} of $G$ over $\mathbb{F}$ is the dimension of the well-covered  vector space of $G$. We will denote this number by $wcdim(G,\mathbb{F})$.
\end{definition}

The concept of the well-covered dimension of a graph was first studied by Caro, Ellingham, Ramey, and Yuster in \cite{CER} and \cite{CY}, and was later developed further by Brown and Nowakowski \cite{BN} in 2005. Some more work has been done on this topic recently; see \cite{BKMUV} and \cite{HOV} for new results only involving graph theory, and see  \cite{C} for an attempt to apply the well-covered dimension of a graph to the study of point-line configurations.

The technique we will use to compute the well-covered dimension of a given graph is fairly natural. Assuming that we know that $M_1, M_2,..., M_n$ are all the maximal independent sets of $G$, then the well-covered weightings of $G$ are given by the
solutions to the linear system of equations:
\begin{align*}
\sum\nolimits_{v \in M_1} x_v &= \sum\nolimits_{v \in M_2} x_v\\
\sum\nolimits_{v \in M_1} x_v &= \sum\nolimits_{v \in M_3} x_v\\
\vdots \\
\sum\nolimits_{v \in M_1} x_v &= \sum\nolimits_{v \in M_n} x_v
\end{align*} 
where $x_v$ denotes the weight of $v\in G(V)$ under a weighting. Since this system may be represented by using a matrix, which we call the \textbf{associated matrix} of $G$ and denote it $A_G$, the well-covered dimension of $G$ may be obtained by
\[
wcdim(G,\mathbb{F}) =  |V(G)| - rank(A_G).
\]

Note that depending on the characteristic of the field $\mathbb{F}$, $rank(A_G)$ could vary. So, for some graphs the well-covered dimension is characteristic-dependent. This is the property we intend to study in this article.

Not much is known about why some graphs have this property. What it is known is that for any given $p$ prime there are graphs that change their well-covered dimension for $char(\mathbb{F})=p$ (see \cite{BN}). In  \cite{BKMUV} it is shown that crown graphs may also be used to construct examples of graphs that change their well-covered dimension for a given $char(\mathbb{F})=p$. Moreover, these graphs are smaller by order than those given in \cite{BN}.

\section{Smallest graphs with characteristic-dependent well-covered dimension}\label{secexamples}

It was already mentioned that examples of graphs with characteristic-dependent well-covered dimension are known. Our focus is on, given a prime number $p$, to  find the smallest (by order) graph such that its well-covered dimension changes for fields with $char(\mathbb{F})=p$. 

We wrote a program in  Sage \cite{Sage} to calculate the well-covered dimension of all graphs of order less or equal to $10$.  Using this program, we were able to perform calculations on many thousands of graphs in order to find the smallest graphs with well-covered dimension that changes for characteristics $2$, $3$, and $5$, respectively. They are:
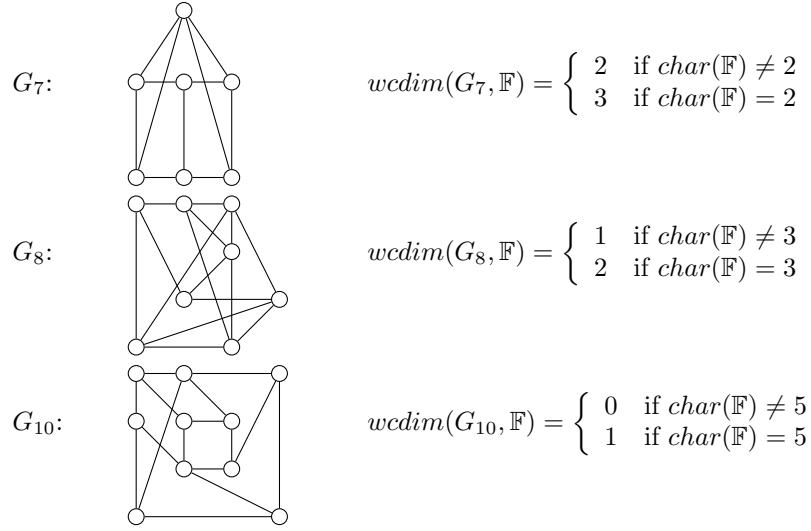
\[
\begin{array}{llll}
\raisebox{.5in}{$G_7$:} & 
\hspace{.2in} \begin{tikzpicture}[x=.25in, y=.25in]
\vertex (v1) at (-1,1) []{};
\vertex (v6) at (0,1) []{};
\vertex (v2) at (1,1) []{};
\vertex (v3) at (-1,-1) []{};
\vertex (v7) at (0,-1) []{};
\vertex (v4) at (1,-1) []{};
\vertex (v5) at (0,2.5) []{};
\path 
(v1) edge (v6)
(v6) edge (v2)
(v6) edge (v7)
(v2) edge (v4)
(v1) edge (v3)
(v3) edge (v7)
(v7) edge (v4)
(v1) edge (v5)
(v2) edge (v5)
(v3) edge (v5)
(v4) edge (v5);
\end{tikzpicture} 
& \ \ \ & 
\raisebox{.5in}{$wcdim(G_7, \mathbb{F}) = \left\{ 
\begin{array}{ll}
2 & \text{if $char(\mathbb{F}) \neq 2$}\\ 
3 &  \text{if $char(\mathbb{F}) = 2$}
\end{array}
\right.$} \\
\raisebox{.5in}{$G_8$:} & 
\hspace{.2in} \begin{tikzpicture}[x=.25in, y=.25in]
\vertex (v1) at (-1,2) []{};
\vertex (v6) at (2,0) []{};
\vertex (v2) at (1,2) []{};
\vertex (v3) at (-1,-1) []{};
\vertex (v4) at (1,-1) []{};
\vertex (v5) at (0,2) []{};
\vertex (v7) at (1,1) []{};
\vertex (v8) at (0,0) []{};
\path 
(v3) edge (v1)
(v3) edge (v2)
(v3) edge (v4)
(v5) edge (v1)
(v5) edge (v2)
(v5) edge (v4)
(v6) edge (v4)
(v6) edge (v3)
(v6) edge (v2)
(v7) edge (v2)
(v7) edge (v4)
(v7) edge (v5)
(v1) edge (v8)
(v6) edge (v8)
(v7) edge (v8);
\end{tikzpicture} 
& \ \ \ & 
\raisebox{.5in}{$wcdim(G_8, \mathbb{F}) = \left\{ 
\begin{array}{ll}
1 & \text{if $char(\mathbb{F}) \neq 3$}\\ 
2 &  \text{if $char(\mathbb{F}) = 3$}
\end{array}
\right.$}\\
\raisebox{.5in}{$G_{10}$:} & 
\hspace{.2in} \begin{tikzpicture}[x=.25in, y=.25in]
\vertex (v1) at (-1,1) []{};
\vertex (v6) at (1,1) []{};
\vertex (v2) at (0,2) []{};
\vertex (v3) at (-1,2) []{};
\vertex (v8) at (2,-1) []{};
\vertex (v5) at (0,1) []{};
\vertex (v7) at (1,0) []{};
\vertex (v4) at (0,0) []{};
\vertex (v9) at (2,2) []{};
\vertex (v10) at (-1,-1) []{};
\path 
(v1) edge (v3)
(v2) edge (v3)
(v1) edge (v4)
(v3) edge (v5)
(v4) edge (v5)
(v2) edge (v6)
(v5) edge (v6)
(v4) edge (v7)
(v6) edge (v7)
(v4) edge (v8)
(v2) edge (v9)
(v7) edge (v9)
(v8) edge (v9)
(v1) edge (v10)
(v2) edge (v10)
(v8) edge (v10);
\end{tikzpicture} 
& \ \ \ & 
\raisebox{.5in}{$wcdim(G_{10}, \mathbb{F}) = \left\{ 
\begin{array}{ll}
0 & \text{if $char(\mathbb{F}) \neq 5$}\\ 
1 &  \text{if $char(\mathbb{F}) = 5$}
\end{array}
\right.$}
\end{array}
\]

We also searched for graphs with order $11$ that had characteristic-dependent well-covered dimension. Based upon our exponentially-increasing search times for graphs of order $7$ through $10$, we estimated that this search could likely take over a week to complete, and so we only obtained partial results for graphs on eleven vertices. We would like to mention that the graphs listed above are not the only ones having characteristic-dependent well-covered dimension. In fact, the analysis of \emph{all} the graphs obtained in our search helped us to conjecture many of the results in the following section. 

\section{Construction of graphs with characteristic-dependent well-covered dimension}

Before getting to the specific constructions, and their proofs, we need to set some notation.

Let $G$ be a graph. For $v\in V(G)$ we define:
\begin{align*}
N(v) & =\{w\in V(G); \  vw \in E(G) \} \\
\overline{N(v)} & =\{v\}\cup N(v) \\
P(v) & = \{ e\in E(G); \ v \ \text{is incident with} \ e \}.
\end{align*}
Also, if $W\subseteq V(G)$ we denote by $G[W]$ the subgraph of $G$ induced by $W$

\begin{lemma}\label{lemlargergraphs}
Let $G$ be a graph.
\begin{enumerate}
\item Assume that $G$ contains a clique $C$ such that $\overline{N(v)}=\overline{N(w)}$, for all $v, w\in C$. Then, $wcdim(G) =wcdim(G')$, where $G'$ is the graph obtained after contracting all the edges in $G[C]$. \\
\item For any given $v\in G$ and $n\in \mathbb{N}$ define $G_{v,n}$ as the graph with 
\begin{align*}
V(G_{v,n}) & =(V(G)\setminus\{v\}) \cup V(K_n)\\
E(G_{v,n}) & = (E(G)\setminus P(v)) \cup \{ab ; \ a\in N(v), \ b\in K_n\}.
\end{align*}
Then, $wcdim(G) =wcdim(G_{v,n})$.
\end{enumerate}
\end{lemma}

\begin{proof}
(1) If $M$ is a maximal independent set of $G$ containing $v\in C$, then we can create a  new  maximal independent set of $G$ by replacing $v$ by any other element $w\in C$ and leaving all other elements in $M$ intact. Hence, when all the vertices in $G[C]$ are contracted the well-covered dimension does not change.\\
(2) This follows directly from (1), as $C=V(K_n)$ is a clique of $G_{v,n}$ satisfying the hypothesis of part (1), and $G$ is the graph obtained from $G_{v,n}$ after contracting $K_n$ to a vertex.
\end{proof}

\begin{remark}
We can now take any of the graphs  in Section \ref{secexamples} and use Lemma \ref{lemlargergraphs} to construct graphs with characteristic-dependent well-covered dimension of any order we may want. For example, using  $G_7$ we can prove that for every $n\geq 7$ there is a graph $G$ of order $n$ for which $wcdim(G, \mathbb{F})$ takes different values depending on whether or not $char(\mathbb{F}) = 2$. 
\end{remark}

We improve part of the previous remark in the following theorem.

\begin{theorem}\label{themGK2upG}
Let $G_{K_2}$ be the graph below.
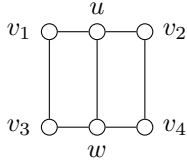
\[
\hspace{.2in} \begin{tikzpicture}[x=.25in, y=.25in]
\vertex (v1) at (-1,1) [label=left:$v_1$]{};
\vertex (u) at (0,1) [label=above:$u$]{};
\vertex (v2) at (1,1) [label=right:$v_2$]{};
\vertex (v3) at (-1,-1) [label=left:${v_3}$]{};
\vertex (w) at (0,-1) [label=below:${w}$]{};
\vertex (v4) at (1,-1) [label=right:${v_4}$]{};
\path 
(v1) edge (u)
(u) edge (v2)
(u) edge (w)
(v2) edge (v4)
(v1) edge (v3)
(v3) edge (w)
(w) edge (v4);
\end{tikzpicture}
\]

Let $G$ be a non-empty graph, and $H_G$ be the disjoint union of $G_{K_2}$ and $G$ with added edges $v_ix$, for all $x \in V(G)$ and  $i=1,2,3,4$. Then, 
\[
wcdim(H_G,\mathbb{F}) = \left\{
\begin{array}{ll}
wcdim(G,\mathbb{F}) + 1 & \text{if $char(\mathbb{F}) \neq 2$} \\
wcdim(G,\mathbb{F}) + 2 & \text{if $char(\mathbb{F}) = 2$}
\end{array}
\right.
\]
\end{theorem}

\begin{proof}
Let $G$ be a non-empty graph of order $m$ where $V(G) = \{x_1,\dots,x_m\}$. We represent $H_G$ as in the figure below,
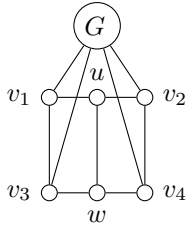
\[
\hspace{.2in} \begin{tikzpicture}[x=.25in, y=.25in]
\vertex (v1) at (-1,1) [label=left:$v_1$]{};
\vertex (u) at (0,1) [label=above:$u$]{};
\vertex (v2) at (1,1) [label=right:$v_2$]{};
\vertex (v3) at (-1,-1) [label=left:${v_3}$]{};
\vertex (w) at (0,-1) [label=below:${w}$]{};
\vertex (v4) at (1,-1) [label=right:${v_4}$]{};
\vertex (v5) at (0,2.5) []{~~$G$\phantom{a}};
\path 
(v1) edge (u)
(u) edge (v2)
(u) edge (w)
(v2) edge (v4)
(v1) edge (v3)
(v3) edge (w)
(w) edge (v4)
(v1) edge (v5)
(v2) edge (v5)
(v3) edge (v5)
(v4) edge (v5);
\end{tikzpicture}
\]

We notice that the maximal independent sets of $H_G$ have the form
\[
\{v_1,v_4\}, \{v_2,v_3\}, \{v_1,v_2,w\}, \{v_3,v_4,u\}, \{u\} \cup M_i, \{w\} \cup M_i 
\] 
where $M_i$ is a maximal independent set of $G$. 

We let $n$ be the number of maximal independent sets of $G$, and for any maximal independent set, $M_i$, of $G$ we define $R_{M_i}$ to be the $m \times 1$ matrix:
\[
\begin{pmatrix}
a_1 & a_2 & \dots & a_{m}
\end{pmatrix}
\]
where $a_j = -1$ if $x_j \in M_i$ and $a_j = 0$ otherwise. Using $\mathbf{0}_m$ for the $m \times 1$ matrix with all $0$ entries, we get that the associated matrix of $H_G$ is:
\[
\begin{pmatrix}
1 & -1 & -1 & 1 & 0 & 0 & \mathbf{0}_m \\
0 & -1 & 0 & 1 & 0 & -1 & \mathbf{0}_m \\
1 & 0 & -1 & 0 & -1 & 0 & \mathbf{0}_m \\
1 & 0 & 0 & 1 & -1 & 0 & R_{M_1} \\
1 & 0 & 0 & 1 & 0 & -1 & R_{M_1} \\
1 & 0 & 0 & 1 & -1 & 0 & R_{M_2} \\
1 & 0 & 0 & 1 & 0 & -1 & R_{M_2} \\
\vdots & \vdots & \vdots & \vdots  & \vdots  & \vdots  & \vdots  \\
1 & 0 & 0 & 1 & -1 & 0 & R_{M_n} \\
1 & 0 & 0 & 1 & 0 & -1 & R_{M_n} 
\end{pmatrix}
\]

By performing row-reduction, we obtain the Hermite normal form of $A_{H_G}$:
\[
\begin{pmatrix}
1 & 0 & 0 & 1 & 0 & 1 & R_{M_1} \\
0 & 1 & 0 &-1 & 0 & 1 & \mathbf{0}_m \\
0 & 0 & 1 & 1 & 0 & 0 & R_{M_1} \\
0 & 0 & 0 & 0 & 1 & 1 & \mathbf{0}_m \\
0 & 0 & 0 & 0 & 0 & 2 & \mathbf{0}_m \\
0 & 0 & 0 & 0 & 0 & 0 & M_{n} - M_{1} \\
\vdots & \vdots & \vdots & \vdots  & \vdots  & \vdots  & \vdots  \\
0 & 0 & 0 & 0 & 0 & 0 & M_{n} - M_{n-1} \\
\end{pmatrix}
\]
We can see that the fifth row of this matrix will have only $0$ entries over a field with characteristic 2, and that the rows below are equivalent to the associated matrix $A_G$.  Hence, the claimed formula for $wcdim(H_G,\mathbb{F})$ is proven.
\end{proof} 

We now take a look at graphs that will have characteristic-dependent well-covered dimension, for odd primes.

\begin{theorem}\label{thmwcdimp}
Let $n\in \mathbb{N}$,  $n\geq 2$. There exists a graph $G$ on $2n+4$ vertices such that: 	
\[
wcdim(G,\mathbb{F}) = \left\{
\begin{array}{cc}
1 & \text{ if } char(\mathbb{F}) \nmid (2n-1) \\
2 & \text{ if } char(\mathbb{F}) \mid (2n-1)
\end{array}
\right.
\]
\end{theorem}

\begin{proof}
Let $G_n$ be the graph constructed from the disjoint union of $K_{n+2}$ and $K_{2,n}$, where $V(K_{n+2}) =\{y_1,y_2,\dots,y_n,v_1,v_2\}$, $V(K_{2,n})=\{w_1,w_2,u_1,u_2,\dots,u_n\}$, and $\{w_1, w_2\}$ is the $2$-vertex part of $K_{2,n}$.  To the edges already present in $K_{n+2}$ and $K_{n,2}$ we add the edges $v_1w_1$, $v_2w_2$, and $y_iu_j$,  for all $i\neq j \in \{1,\dots,n-2\}$.  \\
The following figure shows what $G_3$ looks like.
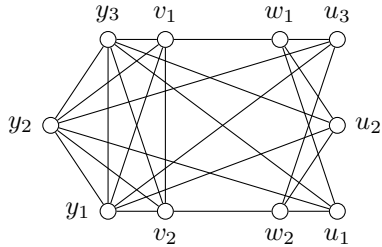
\[
\begin{tikzpicture}[x=.3in, y=.3in]
\vertex (v1) at (-1,-1) [label=left:$y_1$]{};
\vertex (v2) at (-2,0.5) [label=left:${y_2}$]{};
\vertex (v3) at (-1,2) [label=above:$y_3$]{};
\vertex (v4) at (0,-1) [label=below:$v_2$]{};
\vertex (v5) at (0,2) [label=above:${v_1}$]{};
\vertex (v6) at (2,2) [label=above:$w_1$]{};
\vertex (v7) at (3,0.5) [label=right:${u_2}$]{};
\vertex (v8) at (3,-1) [label=below:${u_1}$]{};
\vertex (v9) at (3,2) [label=above:$u_3$]{};
\vertex (v10) at (2,-1) [label=below:$w_{2}$]{};
\path 
(v1) edge (v2)
(v1) edge (v3)
(v1) edge (v4)
(v1) edge (v5)
(v2) edge (v3)
(v2) edge (v4)
(v2) edge (v5)
(v3) edge (v4)
(v3) edge (v5)
(v4) edge (v5)
(v4) edge (v10)
(v5) edge (v6)
(v7) edge (v1)
(v7) edge (v3)
(v8) edge (v2)
(v8) edge (v3)
(v9) edge (v1)
(v9) edge (v2)
(v6) edge (v7)
(v6) edge (v8)
(v6) edge (v9)
(v10) edge (v7)
(v10) edge (v8)
(v10) edge (v9);
\end{tikzpicture}
\]

It is easy to see that $G_n$ has exactly $2n+4$ maximal independent sets. They are:
\begin{enumerate}
\item[(i)] $M_i = \{y_i, w_1, w_2\}$, for all $i \in \{1,\dots,n\}$, 
\item[(ii)]  $N_i = \{y_i, u_i\}$, for all $i \in \{1,\dots,n\}$, 
\item[(iii)]  $L_1=\{v_1, w_2\}$, 
\item[(iv)]  $L_2 = \{v_2, w_1\}$, 
\item[(v)] $K_1 = \{v_1, u_1,\dots,u_n\}$, and 
\item[(vi)] $K_2 = \{v_2, u_1,\dots,u_n\}$.
\end{enumerate}

These maximal independent sets yield the following $2n+3$ linear equations:
\[
\sum\nolimits_{v \in M_i} x_v = \sum\nolimits_{v \in M_{i+1}} x_v \hspace{1in}  \sum\nolimits_{v \in N_i} x_v = \sum\nolimits_{v \in N_{i+1}} x_v 
\]
for $i=1,\ldots,n-1$, and
\begin{align*}
\sum\nolimits_{v \in M_n} x_v &= \sum\nolimits_{v \in N_1} x_v \\
\sum\nolimits_{v \in N_n} x_v &= \sum\nolimits_{v \in L_1} x_v \\
\sum\nolimits_{v \in L_1} x_v &= \sum\nolimits_{v \in L_2} x_v \\
\sum\nolimits_{v \in L_2} x_v &= \sum\nolimits_{v \in K_1} x_v \\
\sum\nolimits_{v \in K_1} x_v &= \sum\nolimits_{v \in K_2} x_v
\end{align*} 
where $x_v$ denotes the weight of $v\in G(V)$ under a weighting.

In order to describe $A_{G_n}$ we need some notation. We will write $M_1-M_2$ to denote the row given by the equation $\sum\nolimits_{v \in M_1} x_v - \sum\nolimits_{v \in M_2} x_v = 0$. We will follow a similar notation for all other rows. Also, we let $M$ to be the following $(n-1)\times n$ matrix
\[
M=
\left(
\begin{array}{cccccc}
 1 & -1 & 0 & \dots & 0 & 0  \\
0 & 1 & -1 & \dots & 0 & 0  \\
\vdots & &\ddots   & \ddots &  &  \vdots  \\
 0 & 0 & 0 & \dots & -1 & 0  \\
0 & 0 & 0 & \dots & 1 & -1  \\
\end{array}
\right)
\]

We would like to clarify to the reader that  we are placing labels around the matrix, they are there to show were each entry in the matrix comes from. We get:
\[
\begin{array}{rc}
                   & \begin{array}{ccc|ccc|cccc}   y_1   & \dots \ & y_n  & u_1 \ & \dots \  & u_n \ & v_1 & v_2 & w_1 & w_2\end{array}\\
                   & \\
\begin{array}{r}
M_1-M_2  \\
\vdots \hspace{.3in} \\ 
M_{n-1}-M_n \\
\hline 
N_n - L_1  \\ 
\hline 
M_n-N_1 \\
\hline 
N_1-N_2 \\
\vdots \hspace{.3in} \\ 
N_{n-1}-N_n \\
\hline 
L_1 - L_2 \\
L_2 - K_1 \\
K_1- K_2 
\end{array} & 
\left(
\begin{array}{ccc|ccc|cccc}
 & \phantom{\dots} &  &  & \phantom{\dots} &  &  & \phantom{\dots} &  \\
\phantom{\vdots} & M &  &  & 0 &  &  & & 0 &  \\
 &  &  &  &  &  &  & &  \\
\hline
0 & \mathbf{0}_{n-2} & 1 &  0 & \mathbf{0}_{n-2}  & 1 & -1 & 0 & 0 & -1 \\
\hline
-1 & \mathbf{0}_{n-2} & 1 & -1 & \mathbf{0}_{n-2}  & 0 & 0 & 0 & 1 & 1  \\
\hline
\phantom{u_n} &  &  \phantom{u_n} &  &  &  &  &  &  &  \\
\phantom{\vdots} & M &  &  & M &  &  &  & 0 & \\
 &  &  &  &  &  &  &  &  &  \\
\hline
 &  &  &  & \mathbf{0}_{n} &  & 1 & -1 & -1 & 1\\
 & 0 &  &  & (-1)_n &  & -1 & 1 & 1 & 0 \\
 &  &  &  & \mathbf{0}_{n} &  & 1 & -1 & 0 & 0 
\end{array}
\right)
\end{array}
\]

We now subtract the first row-block from the fourth. We also notice that the sum of the rows in $M$ add up to $(1, \mathbf{0}_{n-2}, -1)$. So, we add all the rows in the first row-block to the row $M_n-N_1$, and we get
\[
\begin{array}{rc}
                   & \begin{array}{ccc|ccc|cccc}   y_1   & \dots \ & y_n  & u_1 \ & \dots \  & u_n \ & v_1 & v_2 & w_1 & w_2\end{array}\\
                   & \\
\begin{array}{r}
M_1-M_2  \\
\vdots \hspace{.3in} \\ 
M_{n-1}-M_n \\
\hline 
N_n - L_1  \\ 
\hline 
M_n-N_1 \\
\hline 
N_1-N_2 \\
\vdots \hspace{.3in} \\ 
N_{n-1}-N_n \\
\hline 
L_1 - L_2 \\
L_2 - K_1 \\
K_1- K_2 
\end{array} & 
\left(
\begin{array}{ccc|ccc|cccc}
 & \phantom{\dots} &  &  & \phantom{\dots} &  &  & \phantom{\dots} &  \\
\phantom{\vdots} & M &  &  & 0 &  &  & & 0 &  \\
 &  &  &  &  &  &  & &  \\
\hline
0 & \mathbf{0}_{n-2} & 1 &  0 & \mathbf{0}_{n-2}  & 1 & -1 & 0 & 0 & -1 \\
\hline
 & \mathbf{0}_{n} &  & -1 & \mathbf{0}_{n-2}  & 0 & 0 & 0 & 1 & 1  \\
\hline
\phantom{u_n} &  &  \phantom{u_n} &  &  &  &  &  &  &  \\
\phantom{\vdots} & 0 &  &  & M &  &  &  & 0 & \\
 &  &  &  &  &  &  &  &  &  \\
\hline
 &  &  &  & \mathbf{0}_{n} &  & 1 & -1 & -1 & 1\\
 & 0 &  &  & (-1)_n &  & -1 & 1 & 1 & 0 \\
 &  &  &  & \mathbf{0}_{n} &  & 1 & -1 & 0 & 0 
\end{array}
\right)
\end{array}
\]

Since the sum of the rows in $M$ add up to $(1, \mathbf{0}_{n-2}, -1)$, we add all the rows in the third row-block to the row $N_n-L_1$, and we also add the row $M_n-N_1$. We get.
\[
\begin{array}{rc}
                   & \begin{array}{ccc|ccc|cccc}   y_1  \ & \dots \ & y_n \ & u_1 \ & \dots \  & u_n \ & v_1 & v_2 & w_1 & w_2\end{array}\\
                   & \\
\begin{array}{r}
M_1-M_2  \\
\vdots \hspace{.3in} \\ 
M_{n-1}-M_n \\
\hline 
N_n - L_1  \\ 
\hline 
M_n-N_1 \\
\hline 
N_1-N_2 \\
\vdots \hspace{.3in} \\ 
N_{n-1}-N_n \\
\hline 
L_1 - L_2 \\
L_2 - K_1 \\
K_1- K_2 
\end{array} & 
\left(
\begin{array}{ccc|ccc|cccc}
 & \phantom{\dots} &  &  & \phantom{\dots} &  &  & \phantom{\dots} &  \\
\phantom{\vdots} & M &  &  & 0 &  &  & & 0 &  \\
 &  &  &  &  &  &  & &  \\
\hline
0 & \mathbf{0}_{n-2} & 1 &  & \mathbf{0}_{n}  &  & -1 & 0 & 1 & 0 \\
\hline
 & \mathbf{0}_{n} &  & -1 & \mathbf{0}_{n-2}  & 0 & 0 & 0 & 1 & 1  \\
\hline
\phantom{u_n} &  &  \phantom{u_n} &  &  &  &  &  &  &  \\
\phantom{\vdots} & 0 &  &  & M &  &  &  & 0 & \\
 &  &  &  &  &  &  &  &  &  \\
\hline
 &  &  &  & \mathbf{0}_{n} &  & 1 & -1 & -1 & 1\\
 & 0 &  &  & (-1)_n &  & -1 & 1 & 1 & 0 \\
 &  &  &  & \mathbf{0}_{n} &  & 1 & -1 & 0 & 0 
\end{array}
\right)
\end{array}
\]

It is not hard to see that, after row-reduction, blocks $(1,1)$ and $(2,1)$ yield an identity matrix, which given the zeros on blocks $(4,1)$ and $(5,1)$ means that we have maximal rank there. Thus, we do not care about what is on block $(1,3)$.

Now consider blocks $(3,2)$ and $(4,2)$. We see that we can add the top row of this double-block to the one right below to get a $1$ cancelled out. After that we can add this modified row to the one right below, producing the same effect. In this way blocks  $(3,2)$ and $(4,2)$ are row-reduced to a negative identity. Now, these row-reductions affect also block $(4,3)$. We get
\[
\begin{array}{rc}
                   & \begin{array}{ccc|ccc|cccc}  \ \   y_1  & \dots & y_n  & u_1  & \dots   & u_n  &  \ v_1 \ \ \ \ & v_2 \ \ \ \ & w_1 \ \ \ \ & w_2 \ \ \ \ \end{array}\\
                   & \\
\begin{array}{r}
M_1-M_2  \\
\vdots \hspace{.3in} \\ 
M_{n-1}-M_n \\
N_n - L_1  \\ 
\hline 
M_n-N_1 \\
N_1-N_2 \\
\vdots \hspace{.3in} \\ 
N_{n-1}-N_n \\
\hline 
L_1 - L_2 \\
L_2 - K_1 \\
K_1- K_2 
\end{array} & 
\left(
\begin{array}{ccc|ccc|cccc}
 & \phantom{\dots} &  &  & \phantom{\dots} &  &  & \phantom{\dots} &  \\
\phantom{\vdots} & I_n &  &  & 0 &  &  & ? & ? &  \\
 &  &  &  &  &  &  & &  \\
 &  &  &  &  &  &  & &  &  \\
\hline
 &  &  &  &   &  &  &  &  &   \\
\phantom{u_n} &  &  \phantom{u_n} &  &  &  &  &  &  &  \\
\phantom{\vdots} & 0 &  & \ \  & -I_n &  \   & (\mathbf{0}_{n})^t & (\mathbf{0}_{n})^t & (\mathbf{1}_{n})^t & (\mathbf{1}_{n})^t \\
 &  &  &  &  &  &  &  &  &  \\
\hline
 &  &  &  & \mathbf{0}_{n} &  & 1 & -1 & -1 & 1\\
 & 0 &  &  & (-1)_n &  & -1 & 1 & 1 & 0 \\
 &  &  &  & \mathbf{0}_{n} &  & 1 & -1 & 0 & 0 
\end{array}
\right)
\end{array}
\]

We now subtract all the rows of the second row-block from row $L_2-K_1$. We also add row $K_1-K_2$ to $L_2-K_1$, and subtract row $K_1-K_2$ from $L_1-L_2$. We get
\[
\begin{array}{rc}
                   & \begin{array}{ccc|ccc|cccc}  \ \   y_1  & \dots & y_n  & u_1  & \dots   & u_n  &  \ v_1 \ \ \ \ & v_2 \ \ \ \ & w_1 \ \ \ \ & w_2 \ \ \ \ \end{array}\\
                   & \\
\begin{array}{r}
M_1-M_2  \\
\vdots \hspace{.3in} \\ 
M_{n-1}-M_n \\
N_n - L_1  \\ 
\hline 
M_n-N_1 \\
N_1-N_2 \\
\vdots \hspace{.3in} \\ 
N_{n-1}-N_n \\
\hline 
L_1 - L_2 \\
L_2 - K_1 \\
K_1- K_2 
\end{array} & 
\left(
\begin{array}{ccc|ccc|cccc}
 & \phantom{\dots} &  &  & \phantom{\dots} &  &  & \phantom{\dots} &  \\
\phantom{\vdots} & I_n &  &  & 0 &  &  & ? & ? &  \\
 &  &  &  &  &  &  & &  \\
 &  &  &  &  &  &  & &  &  \\
\hline
 &  &  &  &   &  &  &  &  &   \\
\phantom{u_n} &  &  \phantom{u_n} &  &  &  &  &  &  &  \\
\phantom{\vdots} & 0 &  & \ \  \ & -I_n &  \  \ & (\mathbf{0}_{n})^t & (\mathbf{0}_{n})^t & (\mathbf{1}_{n})^t & (\mathbf{1}_{n})^t \\
 &  &  &  &  &  &  &  &  &  \\
\hline
 &  &  &  &  &  & 0 & 0 & -1 & -1\\
 & 0 &  &  & 0 &  & 0 & 0 & n & n-1 \\
 &  &  &  &  &  & 1 & -1 & 0 & 0 
\end{array}
\right)
\end{array}
\]

Finally, adding ($n$ times) row $L_1-L_2$ to $L_2-K_1$ we get
\[
\begin{array}{rc}
                   & \begin{array}{ccc|ccc|cccc}  \ \   y_1  & \dots & y_n  & u_1  & \dots   & u_n  &  \ v_1 \ \ \ \ & v_2 \ \ \ \ & w_1 \ \ \ \ & w_2 \ \ \ \ \end{array}\\
                   & \\
\begin{array}{r}
M_1-M_2  \\
\vdots \hspace{.3in} \\ 
M_{n-1}-M_n \\
N_n - L_1  \\ 
\hline 
M_n-N_1 \\
N_1-N_2 \\
\vdots \hspace{.3in} \\ 
N_{n-1}-N_n \\
\hline 
L_1 - L_2 \\
L_2 - K_1 \\
K_1- K_2 
\end{array} & 
\left(
\begin{array}{ccc|ccc|cccc}
 & \phantom{\dots} &  &  & \phantom{\dots} &  &  & \phantom{\dots} &  \\
\phantom{\vdots} & I_n &  &  & 0 &  &  & ? & ? &  \\
 &  &  &  &  &  &  & &  \\
 &  &  &  &  &  &  & &  &  \\
\hline
 &  &  &  &   &  &  &  &  &   \\
\phantom{u_n} &  &  \phantom{u_n} &  &  &  &  &  &  &  \\
\phantom{\vdots} & 0 &  & \ \  \ & -I_n &  \  \ & (\mathbf{0}_{n})^t & (\mathbf{0}_{n})^t & (\mathbf{1}_{n})^t & (\mathbf{1}_{n})^t \\
 &  &  &  &  &  &  &  &  &  \\
\hline
 &  &  &  &  &  & 0 & 0 & -1 & -1\\
 & 0 &  &  & 0 &  & 0 & 0 & 0 & 2n-1 \\
 &  &  &  &  &  & 1 & -1 & 0 & 0 
\end{array}
\right)
\end{array}
\]

We can see that this matrix has full rank $(2n+3)$ over a field with characteristic $p \nmid (2n-1)$ and that the previous to last row becomes a row of all zeroes over a field with characteristic $p \mid (2n-1)$. So, since $G_n$ has $(2n+4)$ vertices, we get
\[
wcdim(G_n,\mathbb{F}) = |V(G_n)| - rank(A_{G_n}) = (2n+4) - (2n+3) = 1 \text{ if } char(\mathbb{F}) \nmid (2n-1), 
\]
and
\[
wcdim(G_n,\mathbb{F}) = |V(G_n)| - rank(A_{G_n}) = (2n+4) - (2n+2) = 2 \text{ if } char(\mathbb{F}) \mid (2n-1).
\]
which is what we wanted to show.
\end{proof}

\begin{corollary}
Let $p$ be a prime. For every $n\geq p+5$ there is a graph $G$ of order $n$ for which $wcdim(G, \mathbb{F})$ takes different values depending on whether $char(\mathbb{F})$ is $p$ or not. 
\end{corollary}

\begin{proof}
This is an immediate corollary of Lemma \ref{lemlargergraphs} and Theorems \ref{themGK2upG} and \ref{thmwcdimp}.
\end{proof}

The previous theorem and corollary can be used to construct the smallest known graphs which have characteristic-dependent well-covered dimension for any given characteristic. Based upon the many graphs generated by our Sage script, we speculate that these graphs constructed using the previous theorem are the smallest graphs that exist with this property. We leave this as an open problem.

The hope is that having found potentially the smallest graphs with characteristic-dependent well-covered dimension for a given characteristic $p$, we can more easily find other properties of the graphs which cause the well-covered dimension of the graph to be characteristic-dependent. We believe this is also a very interesting open problem.


\end{document}